\documentclass[preprint,12pt]{elsarticle}

\usepackage{amssymb}

\newtheorem{thm}{Theorem}

\newdefinition{dfn}{Definition}
\newdefinition{ex}{Example}
\newproof{proof}{Proof}

\begin{document}

\begin{frontmatter}



\title{Asymptotic formulas for eigenvalues and eigenfunctions  of a new boundary-value-transmission problem }


\author[rvt]{O. Sh. Mukhtarov\corref{cor1}}
\ead{omukhtarov@yahoo.com} \cortext[cor1]{Corresponding Author (Tel:
+90 356 252 16 16, Fax: +90 356 252 15 85)}
\author[rvt]{K. Aydemir}
\ead{kadriye.aydemir@gop.edu.tr}

\address[rvt]{Department of Mathematics, Faculty of Arts and Science, Gaziosmanpa\c{s}a University,\\
 60250 Tokat, Turkey}

\begin{abstract}
In this paper we are concerned with a new class of BVP' s consisting
of eigendependent boundary conditions and two supplementary
transmission conditions at one interior point. By modifying some
techniques of classical Sturm-Liouville theory
 and  suggesting own approaches we find asymptotic formulas for the eigenvalues and eigenfunction.
\end{abstract}

\begin{keyword}
Sturm-Liouville problems, eigenvalue, eigenfunction, asymptotics of
eigenvalues and eigenfunction.


\end{keyword}

\end{frontmatter}


\section{Introduction}
Boundary value problems arise directly as mathematical models of
motion according to Newton's law, but more often as a result of
using the method of separation of variables to solve the classical
partial differential equations of physics, such as Laplace's
equation, the heat equation, and the wave equation. Many topics in
mathematical physics require investigations of eigenvalues and
eigenfunctions of boundary value problems.  These investigations are
of utmost importance for theoretical and applied problems in
mechanics, the theory of vibrations and stability, hydrodynamics,
elasticity, acoustics, electrodynamics, quantum mechanics, theory of
systems and their optimization, theory of random processes, and many
other branches of natural science. Such problems are formulated in
many different ways.

In this study we shall investigate a new class of Sturm-Liouville
type problem which consist of a Sturm-Liouville equation contained
\begin{equation}\label{1}
\mathcal{L}(y):=-p(x)y^{\prime \prime }(x)+ q(x)y(x)=\mu^{2}y(x)
\end{equation}
to hold in finite interval $(a, b)$ except at one inner point $c \in
(a, b) $, where discontinuity in $u  \ \textrm{and} \ u'$   are
prescribed by two-point eigenparameter- dependent boundary
conditions
\begin{equation}\label{2}
V_{1}(y):=\alpha_{10}y(a)-\alpha_{11}y'(a)-\mu^{2}(\alpha'_{10}y(a)-\alpha'_{11}y'(a))=0,
\end{equation}
\begin{equation}\label{3}
V_{2}(y):=\alpha_{20}y(b)-\alpha_{21}y'(b)+\mu^{2}(\alpha'_{20}y(b)-\alpha'_{21}y'(b))=0,
\end{equation}
 together with the transmission conditions
\begin{equation}\label{4}
V_{j}(y):=\beta^{-}_{j1}y'(c-)+\beta^{-}_{j0}y(c-)+\beta^{+}_{j1}y'(c+)+\beta^{+}_{j0}y(c+)=0,
\ \ j=1,2
\end{equation}
where $p(x)=p^{-}>0 \ \textrm{for} \ x \in [a, c) $, $p(x)=p^{+}>0 \
\textrm{for} \ x \in (c, b], $ the potential $q(x)$ is real-valued
function which continuous in each of the intervals $[a, c) \
\textrm{and} \ (c, b]$, and has a finite limits $q( c\mp0)$,
$\lambda$ \ is a physical  complex parameter, \ $\alpha_{ij}, \ \
\beta^{\pm}_{ij}, \ \alpha'_{ij} \ (i=1,2 \ \textrm{and} \ j=0,1)$
are real numbers. We describe some analytical solutions of the
problem and find asymptotic formulas of eigenvalues and
eigenfunctions. These boundary conditions are of great importance
for theoretical and applied studies and have a definite mechanical
or physical meaning (for instance, of free ends). Also the problems
with transmission conditions  arise in mechanics, such as thermal
conduction problems for a thin laminated plate, which studied in
\cite{tik}. This class of problems essentially differs from the
classical case, and its investigation requires a specific approach
based on the method of separation of variables. Moreover the
eigenvalue parameter appear in one of the boundary conditions and
two new conditions added to boundary conditions called transmission
conditions.
\section{\textbf{The fundamental solutions and characteristic
Function } }
 Let $B_{0}= \left[%
\begin{array}{cccc}
  \alpha_{11} & \alpha_{10}  \\
  \alpha'_{11} & \alpha'_{10}
  \\
\end{array} %
 \right]$,  \  $B_{1}= \left[%
\begin{array}{cccc}
  \alpha_{21} & \alpha_{20}  \\
  \alpha'_{21} & \alpha'_{20}
  \\
\end{array} %
 \right]$ \ and T=$ \left[%
\begin{array}{cccc}
  \beta^{-}_{10} & \beta^{-}_{11} & \beta^{+}_{10} & \beta^{+}_{11} \\
  \beta^{-}_{20} & \beta^{-}_{21} & \beta^{+}_{20} & \beta^{+}_{21}
  \\
\end{array} %
 \right]. $
 Denote  the determinant of the matrix $B_{0}$ by $\theta_{1}$, the determinant of the matrix $B_{1}$ by $\theta_{2}$  and the determinant of the k-th
 and
j-th columns of the matrix T  by $\Delta_{kj}$. Note that throughout
this study we shall assume that $ \theta_{1}>0, \ \theta_{2}>0, \
\Delta_{12}>0 \ \ \textrm{and} \ \Delta_{34}>0.$
 With a view to constructing the
characteristic function we define  two solution $\psi(x,\mu) \
\textrm{and}  \ \varphi(x,\mu)$ as follows. Denote the solutions of
the equation \ref{1} satisfying the initial conditions
\begin{equation}\label{7}
y(a)=\alpha _{11}-\mu^{2}\alpha' _{11}, \ y^{\prime }(a)=\alpha
_{10}-\mu^{2}\alpha' _{10}
\end{equation}
and
\begin{equation}\label{10}
y(b)=\alpha _{21}+\mu^{2}\alpha' _{21}, \ y^{\prime }(b)=\alpha
_{20}+\mu^{2}\alpha' _{20}
\end{equation}
by $u=\varphi^{-}(x,\mu)$ and $u=\psi^{+}(x,\mu),$ respectively. It
is known that the initial-value problems has an unique solutions \
$u=\varphi^{-}(x,\mu)$ and $u=\psi^{+}(x,\mu),$ \ which is an entire
function of $\mu \in \mathbb{C}$ for each fixed $x \in [a,c)$  and
$x \in (c,b]$ respectively. (see, for example, \cite{titc}). Using
this solutions we can prove that the equation (\ref{1}) on $ [a,c)$
and $x \in (c,b]$ has solutions $u=\varphi^{+}(x,\mu)$ and
$u=\psi^{-}(x,\mu),$ satisfying the initial conditions
\begin{eqnarray}\label{8}
&&y(c) =\frac{1}{\Delta_{12}}(\Delta_{23}\varphi^{-}(c,\mu
)+\Delta_{24}\frac{\partial\varphi^{-}(c,\mu )}{\partial x})\\
&& \label{9} y^{\prime }(c)
=\frac{-1}{\Delta_{12}}(\Delta_{13}\varphi ^{-}(c,\mu
)+\Delta_{14}\frac{\partial\varphi^{-}(c,\mu )}{\partial x}).
\end{eqnarray}
and
\begin{eqnarray}\label{11}
&&y(c) =\frac{-1}{\Delta_{34}}(\Delta_{14}\psi^{+}(c,\mu
)+\Delta_{24}\frac{\partial\psi^{+}(c,\mu )}{\partial x}),\\
&& \label{12} y^{\prime }(c) =\frac{1}{\Delta_{34}}(\Delta_{13}\psi
^{+}(c,\mu )+\Delta_{23}\frac{\partial\psi^{+}(c,\mu )}{\partial
x}).
\end{eqnarray}
respectively. Consequently, the solution $u=\varphi(x,\mu)$ defined
by
\begin{eqnarray} \varphi(x,\mu)=\left
\{\begin{array}{ll}\label{91}
\varphi^{-}(x,\mu), & x\in \lbrack a,c) \\
\varphi^{+}(x,\mu), & x\in (c,b\rbrack \\
\end{array}\right.
\end{eqnarray}
 satisfies the equation (\ref{1}) on  $[a,c) \cup (c,b]$, the
 first boundary condition  (\ref{2}) and both transmission
 conditions
 (4).\\
 Similarly,  the solution $\psi(x,\mu)$ defined by
\begin{eqnarray} \psi(x,\mu)=\left
\{\begin{array}{ll}\label{92}
\psi^{-}(x,\mu), & x\in \lbrack a,c) \\
\psi^{+}(x,\mu), & x\in (c,b\rbrack  \\
\end{array}\right.
\end{eqnarray}
 satisfies the equation (\ref{1}) on
whole $[a,c) \cup (c,b]$, the second boundary condition of (\ref{3})
and both transmission condition  (\ref{4}).
 Again, similarly to
\cite{os1} it can be proven that, there are infinitely many
eigenvalues $\mu_{n}, \ n=1,2,...$   of the BVTP
$(\ref{1})-(\ref{4})$ which are coincide with the zeros of
characteristic function  \ $w(\mu)$.

\section{  Some asymptotic approximation formulas
for fundamental solutions}
 Below, for
shorting  we shall use also  notations; $\varphi ^{\pm}(x,\mu
):=\varphi^{\pm} _{\mu}(x), \psi ^{\pm}(x,\mu
):=\psi^{\pm}_{\mu}(x).$ By applying the method of variation of
parameters we can prove that the next integral and
integro-differential equations are hold for $k=0$ and $k=1.$%
\begin{eqnarray*}\label{(4.2)}
&&\frac{d^{k}}{dx^{k}}\varphi ^{-}_{\mu}(x ) =
\sqrt{p^{-}}\frac{(\alpha _{10}-\mu^{2}\alpha'
_{10})}{\mu}\frac{d^{k}}{dx^{k}}\sin \left[\frac{\mu\left(
x-a\right)}{\sqrt{p^{-}}}\right]\nonumber \\
&&+(\alpha _{11}-\mu^{2}\alpha' _{11})\frac{d^{k}}{dx^{k}}\cos
\left[\frac{\mu\left( x-a\right)}{\sqrt{p^{-}}}\right]
+\frac{1}{\sqrt{p^{-}}\mu}\int\limits_{a}^{x}\frac{d^{k}}{dx^{k}}\sin
\left[\frac{\mu\left( x-z\right)}{\sqrt{p^{-}}}\right] q(z)\varphi
^{-}_{\mu}(z)dz
\end{eqnarray*}
\begin{eqnarray}
\frac{d^{k}}{dx^{k}}\psi ^{-}_{\mu}(x )
&=&-\frac{1}{\Delta_{34}}(\Delta_{14}\psi ^{+}(c,\mu )+\Delta_{24}
\frac{\partial\psi^{+}(c,\mu )}{\partial x}))
\frac{d^{k}}{dx^{k}}\cos \left[
\frac{\mu(x-c)}{\sqrt{p^{-}}}\right] \nonumber \\
&&+\frac{\sqrt{p^{-}}}{\mu \Delta_{34}}(\Delta_{13}\psi ^{+}(c,\mu
)+\Delta_{23}\frac{\partial\psi^{+}(c,\mu )}{\partial
x}))\frac{d^{k}}{dx^{k}}\sin \left[ \frac{\mu(x-c)}{\sqrt{p^{-}}}
\right]  \nonumber  \\
&&+\frac{1}{\sqrt{p^{-}}\mu}\int\limits_{x}^{c-}\frac{d^{k}}{dx^{k}}\sin
\left[ \mu\left( x-z\right) \right] q(z)\psi ^{-}_{\mu}(z)dz
\label{(4.22)}
\end{eqnarray}
for $x \in [a,c)$ and
\begin{eqnarray}
\frac{d^{k}}{dx^{k}}\varphi^{+}_{\mu}(x )
&=&\frac{1}{\Delta_{12}}(\Delta_{23}\varphi ^{-}(c,\mu
)+\Delta_{24}\frac{\partial\varphi^{-}(c,\mu )}{\partial x}))
\frac{d^{k}}{dx^{k}}\cos \left[
\frac{\mu(x-c)}{\sqrt{p^{+}}}\right]  \nonumber  \\
&&-\frac{\sqrt{p^{+}}}{\mu \Delta_{12}}(\Delta_{13}\varphi
^{-}(c,\mu )+\Delta_{14}\frac{\partial\varphi^{-}(c,\mu )}{\partial
x}))\frac{d^{k}}{dx^{k}}\sin \left[ \frac{\mu(x-c)}{\sqrt{p^{+}}}
\right]  \nonumber \\
&&+\frac{1}{\sqrt{p^{+}}\mu}\int\limits_{c+}^{x}\frac{d^{k}}{dx^{k}}\sin
\left[ \frac{\mu\left( x-z\right)}{\sqrt{p^{+}}} \right] q(z)\varphi
^{+}_{\mu}(z)dz \label{(4.a)}
\end{eqnarray}
\begin{eqnarray*}
&&\frac{d^{k}}{dx^{k}}\psi^{+} _{\mu}(x )=
\frac{\sqrt{p^{+}}}{\mu}(\alpha _{20}+\mu^{ 2}\alpha'
_{20})\frac{d^{k}}{dx^{k}}\sin \left[ \frac{\mu\left(
x-b\right)}{\sqrt{p^{+}}} \right] \nonumber\\&&+(\alpha
_{21}+\mu^{2} \alpha' _{21})\frac{d^{k}}{dx^{k}} \cos \left[
\frac{\mu\left( x-b\right)}{\sqrt{p^{+}}} \right]
+\frac{1}{\sqrt{p^{+}}\mu}\int\limits_{x}^{b}\frac{d^{k}}{dx^{k}}\sin
\left[\frac{\mu\left( x-z\right)}{\sqrt{p^{+}}} \right] q(z)\psi
^{+}_{\mu}(z)dz \label{(4.21)}
\end{eqnarray*}
for $x \in (c,b]$. Now we are ready to prove the following theorems.

\begin{thm} \label{(4.n)}
Let , $Im \mu=t.$ Then \ if $\alpha' _{11}\neq 0$ the estimates
\begin{eqnarray}
\frac{d^{k}}{dx^{k}}\varphi ^{-}_{\mu}(x ) &=&-\alpha' _{11}\mu^{2}\frac{d^{k}}{dx^{k}}%
\cos \left[ \frac{\mu\left( x-a\right)}{\sqrt{p^{-}}} \right]
+O\left( \left| \mu\right| ^{k+1}e^{\frac{\left| t\right|
(x-a)}{\sqrt{p^{-}}}}\right)
\label{(4.3)} \\
\frac{d^{k}}{dx^{k}}\varphi^{+} _{\mu}(x ) &=&\frac{\Delta_{24}}{%
\Delta_{12}}\frac{\alpha' _{11}}{\sqrt{p^{-}}}\mu^{3}\sin
\left[\frac{ \mu\left( c-a\right) }{\sqrt{p^{-}}}\right]
\frac{d^{k}}{dx^{k}}\cos\left[ \frac{\mu\left( x-c\right)
}{\sqrt{p^{+}}}\right] \nonumber
\\
&&+O\left(|\mu| ^{k+2} e^{\left|
t\right|(\frac{(x-c)}{\sqrt{p^{+}}}+\frac{(c-a)}{\sqrt{p^{-}}})}\right)
\label{(4.4)}
\end{eqnarray}
 are valid as $\left| \mu \right| \rightarrow \infty $, while if $\alpha' _{11}=0$ the estimates
\begin{eqnarray}
\frac{d^{k}}{dx^{k}}\varphi^{-} _{\mu}(x ) &=&-\alpha'_{10}\sqrt{p^{-}}\mu%
\frac{d^{k}}{dx^{k}}\sin \left[ \frac{\mu(x-a)}{\sqrt{p^{-}}}\right]
+O\left( \left| \mu\right| ^{k}e^{\frac{\left| t\right|
(x-a)}{\sqrt{p^{-}}}}\right) \label{(4.5)}
\\
\frac{d^{k}}{dx^{k}}\varphi^{+} _{\mu}(x ) &=&-\frac{\Delta_{24}}{%
\Delta_{12}}\alpha' _{10}\mu^{2}\cos  \left[ \frac{\mu\left( c-a\right)}{\sqrt{p^{-}}} \right] \frac{d^{k}}{dx^{k}} \cos \left[ \frac{\mu(x-c)}{\sqrt{p^{+}}}\right]  \nonumber\\
&&+O\left( \left| \mu\right| ^{k+1}e^{\left|
t\right|(\frac{(x-c)}{\sqrt{p^{+}}}+\frac{(c-a)}{\sqrt{p^{-}}})}\right)
\label{(4.6)}
\end{eqnarray}
are valid as $\left| \mu \right| \rightarrow \infty $ ($k=0,1)$. All
asymtotic estimates  are uniform with respect to $x.$
\end{thm}
\begin{proof}
The asymptotic formulas for $(\ref{(4.3)})$ in $(\ref{(4.4)})$
follows immediately from the Titchmarsh's Lemma on the asymptotic
behavior of $\varphi^{-}
_{\mu }(x)$ (\cite{titc}, Lemma 1.7). But the corresponding formulas for $%
\varphi^{+} _{\mu}(x )$ need individual consideration. Let $\alpha
_{11}\neq 0$. Put $\varphi ^{+}_{\mu}(x )=e^{-\left| t\right| (x-a)}
Y(x,\mu )$ it follows from $(\ref{(4.a)})$ that
\begin{eqnarray}
Y(x,\mu ) &=&\frac{1}{\Delta_{12}}\alpha _{11}e^{-\left| t\right|
(x-a)} \big[\Delta_{23}\cos s(c-a)\cos
s(x-c)-\Delta_{24}s \sin s(c-a)%
\cos s(x-c) \nonumber \\&&-\frac{\Delta_{13}}{s}\cos s (c-a) \sin
s(x-c)+\frac{\Delta_{14}}{s}\sin s(c-a) \sin s(x-c)
\big] \nonumber \\
&&+\frac{1}{s}\int\limits_{c+}^{x}\sin \left[ s(x-z)\right] q(z)
e^{-\left| t\right| (x-z)} Y(z,\mu )dz +O(1)  \label{(4.7)}
\end{eqnarray}
Denoting $ Y(\mu )= \max_{x\in ( c,b]} |Y(x,\mu )| \ \textrm{ and }
\ \widetilde{q}=\int\limits_{c+}^{b}|q(z)|dz $ from the last
equation we have
\begin{eqnarray*}
Y(\mu )\leq |\frac{\Delta_{23}}{\Delta_{12}}\alpha _{11}|+ |\frac{\Delta_{24}}{\Delta_{12}}\alpha _{11}|%
+ |\frac{\Delta_{13}}{\Delta_{12}}\alpha _{11}|+
|\frac{\Delta_{14}}{\Delta_{12}}\alpha
_{11}|+\frac{\widetilde{q}}{a_{2}|s|}Y(\mu )+M
\end{eqnarray*}
for some $M>0$. Consequently $Y(\mu )=O(1)  \ \textrm{as} \left| \mu
\right| \rightarrow \infty ,$ so
\begin{eqnarray} \label{as} \varphi^{+}
_{\mu}(x )=O(e^{\left| t\right| (x-a)}). \end{eqnarray}
Consequently, the estimate (\ref{(4.4)}) for the case $k=0$ are
obtained by substituting $(\ref{as})$ in the integrals on the
right-hand side of $(\ref{(4.22)})$ . The case $k=1$ of the
$(\ref{(4.4)})$ follows at once on differentiating $(\ref{(4.a)})$
and making the same procedure as in the case $k=0$. The proof of
$(\ref{(4.5)})$ in $(\ref{(4.6)})$ is similar.
\end{proof}
Similarly  we can easily obtain the following Theorem for
$\psi^{\pm}_{\mu}(x ).$
\begin{thm} \label{(c1)}
Let $Im \mu=t.$ Then \ if $\alpha' _{21}\neq 0$
\begin{eqnarray}
\frac{d^{k}}{dx^{k}}\psi ^{+}_{\mu}(x ) &=&\alpha' _{21}\mu ^{2}\frac{d^{k}}{dx^{k}}%
\cos \left[ \frac{\mu\left( b-x\right)}{\sqrt{p^{+}}} \right]
+O\left( \left| \mu\right| ^{k+1}e^{\frac{\left| t\right|
(b-x)}{\sqrt{p^{+}}}}\right)
\label{(c2)} \\
\frac{d^{k}}{dx^{k}}\psi^{-} _{\mu}(x ) &=&-\frac{\Delta_{24}}{%
\Delta_{34}}\frac{\alpha' _{21}}{\sqrt{p^{+}}}\mu ^{3}\sin  \left[
\frac{\mu\left( b-c\right)}{\sqrt{p^{+}}} \right]
\frac{d^{k}}{dx^{k}}\cos \left[ \frac{\mu\left(
x-c\right)}{\sqrt{p^{-}}} \right] \nonumber
\\
&&+O\left(|\mu| ^{k+2} e^{\left|
t\right|(\frac{(b-c)}{\sqrt{p^{+}}}+\frac{(c-x)}{\sqrt{p^{-}}})}\right)
\label{(4.n1)}
\end{eqnarray}
as $\left| \mu \right| \rightarrow \infty $, while if $\alpha'_{21}=0$%

\begin{eqnarray}
\frac{d^{k}}{dx^{k}}\psi^{+} _{\mu}(x ) &=&-a'_{20}\sqrt{p^{+}}\mu%
\frac{d^{k}}{dx^{k}}\sin \left[ \frac{\mu(b-x)}{\sqrt{p^{+}}}\right]
+O\left( \left| \mu\right| ^{k}e^{\left| t\right|
\frac{(b-x)}{\sqrt{p^{+}}}}\right) \label{(c3)}
\\
\frac{d^{k}}{dx^{k}}\psi^{-} _{\mu}(x) &=&-\frac{\Delta_{24}}{%
\Delta_{34}}\alpha' _{20}\mu^{2}\cos  \left[ \frac{\mu\left( b-c\right)}{\sqrt{p^{+}}} \right] \frac{d^{k}}{dx^{k}} \cos \left[ \frac{\mu(x-c)}{\sqrt{p^{-}}}\right]  \nonumber \\
&&+O\left( \left| \mu\right| ^{k+1}e^{\left|
t\right|(\frac{(b-c)}{\sqrt{p^{+}}}+\frac{(c-x)}{\sqrt{p^{-}}})}\right)
\label{(c4)}
\end{eqnarray}
as $\left| \mu \right| \rightarrow \infty $ ($k=0,1)$. Each of this
asymptotic equalities hold uniformly for $x.$
\end{thm}
\section{Asymptotic behaviour
 of eigenvalues and corresponding eigenfunctions%
} It is well-known from ordinary differential equation theory that
the Wronskians $W[\varphi^{-}_{\mu}(x ),\psi ^{-}_{\mu}(x )]$ and
$W[\varphi ^{+}_{\mu}(x ),\psi ^{+}_{\mu}(x )]$ are independent of
variable $x.$ By using (\ref{8})-(\ref{9}) and (\ref{11})-(\ref{12})
we have
\begin{eqnarray*}
w^{+}(\mu ) &=&\varphi^{+}(c,\mu )\frac{\partial\psi ^{+}(c,\mu
)}{\partial x}-\frac{\partial\varphi
^{+}(c,\mu )}{\partial x}\psi ^{+}(c,\mu ) \\
 &=&\frac{\Delta_{34}}{\Delta_{12}}(\varphi^{-}(c,\mu )\frac{\partial\psi
^{-}(c,\mu )}{\partial x}-\frac{\partial\varphi
^{-}(c,\mu )}{\partial x}\psi ^{-}(c,\mu )) \\
&=&\frac{\Delta_{34}}{\Delta_{12}} w^{-}(\mu ).
\end{eqnarray*}
Denote $w(\mu):= \Delta_{34} w^{-}(\mu) = \Delta_{12} \ w^{+}(\mu).$
\begin{thm}
The eigenvalues of the problem $(\ref{1})$-$(\ref{4})$ are consist
of the zeros of the function $w(\mu ).$\label{t2}
\end{thm}
\begin{proof}
Let $w(\mu _{0})=0.$ Then $W[\varphi^{-} _{\mu _{0}} ,\psi^{-} _{\mu
_{0}}]_{x}=0$. Thus, the functions $\varphi _{1\mu _{0}}(x)$ and
$\psi^{-} _{\mu _{0}}(x)$ are linearly depended, i.e.,
\begin{equation}
\psi _{1\mu _{0}}(x)=k\varphi _{\mu _{0}}(x),x\in \lbrack -1,0]
\label{(3.20)}
\end{equation}
for some $k\neq 0.$ In view of this equality $\psi_{\mu _{0}} (x)$
satisfies also the first boundary condition $(\ref{2}).$ Recall that
the solution $\psi_{\mu _{0}} (x)$  also satisfies the second
boundary condition $(\ref{3})$ and both transmission conditions
$(\ref{4})$ Consequently $\psi_{\mu _{0}} (x)$ is an eigenfunction
of the problem$(\ref{1})$- $(\ref{4})$ corresponding to the
eigenvalue $\mu _{0}.$ Hence each zero of $w(\mu )$ is an
eigenvalue.

Now let $y_{0}(x)$ be any eigenfunction corresponding to eigenvalue
$\mu _{0}$. Suppose, that $w(\mu _{0})\neq 0$. Then the couples of
the functions $\varphi^{-}, \psi ^{-}$ and $\varphi ^{+}, \psi ^{+}$
would be
linearly independent on $[a,c)$ and $(c,b]$ respectively. Therefore, the solution $%
y_{0}(x)$ may be represented in the form
\begin{equation}
\ \ y_{0}(x)=\left\{
\begin{array}{c}
k_{11}\varphi^{-} _{\mu _{0}}(x)+k_{12}\psi^{-} _{\mu _{0}}(x)\textrm{ for }%
x\in \lbrack a,c) \\
k_{21}\varphi^{+} _{\mu _{0}}(x)+k_{22}\psi^{+} _{\mu _{0}}(x)\textrm{ for }%
x\in c,b]%
\end{array}
\right.  \label{(3.21)}
\end{equation}
where at least one of the coefficients $k_{11}, k_{12}, k_{21}$ and
$k_{22}$ is not zero. Considering the equations
\begin{equation}
\tau_{i}(y_{0}(x))=0, \ \ i=1,2, 3, 4    \label{(3.255)}
\end{equation}
 as the homogenous system of linear equations of the variables $k_{11}, k_{12},
 k_{21}, k_{22}$ and taking into account the conditions
 $(\ref{2})-(\ref{4})$
we obtain homogenous linear simultaneous equation of the variables
$k_{ij}(i,j=1,2)$ the determinant of which is equal to
$\frac{1}{\Delta_{12} \Delta_{34}}\omega^{3}(\mu)$ and therefore
does not vanish by assumption. Consequently this linear simultaneous
equation has the only trivial solution $k_{ij}=0(i,j=1,2).$  We thus
arrive at a  contradiction, which completes the proof.
\end{proof}

Now by modifying the standard method we  prove that all eigenvalues
of the problem $(\ref{1})-(\ref{4})$ are real.
\begin{thm}
The eigenvalues of the boundary value transmission problem
$(\ref{1})-(\ref{4})$ are real.
\end{thm}
\begin{proof}
\end{proof}

Since the Wronskians of $\varphi^{+} _{\mu}(x )$ and $\psi^{+}
_{\mu}(x )$ are independent of $x$, in particular, by putting $x=a$
we have
\begin{eqnarray}\label{(ko)}
\omega(\mu ) &=& \Delta_{34}\{\varphi^{-}(a,\mu
)\frac{\partial\psi^{-}(a,\mu )}{\partial x}-\frac{\partial\varphi
^{-}(a,\mu )}{\partial x}\psi^{-}(a,\mu )\}\nonumber \\
&=& \Delta_{34}\{(\alpha _{11}-\mu^{2}\alpha'
_{11})\frac{\partial\psi^{-}(a,\mu )}{\partial x}+(\alpha
_{10}-\mu^{2}\alpha' _{10})\psi^{-}(a,\mu )\}.
\end{eqnarray}
 Let  $Im \mu=t.$ By substituting $(\ref{(c2)})$ and
$(\ref{(c4)})$ in $(\ref{(ko)})$ we obtain easily the following
asymptotic representations\\ \textbf{(i)} If $\alpha _{21} ^{\prime
}\neq 0$ and $\alpha _{11}^{\prime }\neq 0$, then
\begin{equation}\label{(4.15)}
w(\mu )=-\frac{\Delta_{24}\alpha _{11}^{\prime }\alpha _{21}^{\prime }}{\sqrt{p^{-}}\sqrt{p^{+}}}%
\mu^{6} \sin\left[ \frac{\mu\left( c-a\right)}{\sqrt{p^{-}}}
\right]\sin \left[ \frac{\mu\left( b-c\right)}{\sqrt{p^{+}}} \right]
+O\left( \left| \mu\right| ^{5}e^{\left| t\right|
(\frac{(b-c)}{\sqrt{p^{+}}}+\frac{(c-a)}{\sqrt{p^{-}}} )}\right)
\end{equation}
\textbf{(ii)} If $\alpha _{21} ^{\prime }\neq  0$ and $\alpha
_{11}^{\prime }= 0$, then
\begin{equation}
w(\mu )=-\frac{\Delta_{24}\alpha _{10}^{\prime }\alpha _{21}^{\prime }}{\sqrt{p^{+}}}%
\mu^{5}\cos\left[ \frac{\mu\left( c-a\right)}{\sqrt{p^{-}}}
\right]\sin \left[\frac{ \mu\left( b-c\right)}{\sqrt{p^{+}}} \right]
+O\left( \left| \mu\right| ^{4}e^{\left| t\right|
(\frac{(b-c)}{\sqrt{p^{+}}}+\frac{(c-a)}{\sqrt{p^{-}}} )}\right)
\label{(4.16)}
\end{equation}
\textbf{(iii)} If $\alpha _{21} ^{\prime }= 0$ and $\alpha
_{11}^{\prime }\neq 0$, then
\begin{equation}
w(\mu )=-\frac{\Delta_{24}\alpha _{11}^{\prime }\alpha _{20}^{\prime }}{\sqrt{p^{-}}}%
\mu^{5}\sin \left[ \frac{\mu\left( c-a\right)}{\sqrt{p^{-}}}
\right]\cos \left[ \frac{\mu\left( b-c\right)}{\sqrt{p^{+}}} \right]
+O\left( \left| \mu\right| ^{4}e^{\left| t\right|
(\frac{(b-c)}{\sqrt{p^{+}}}+\frac{(c-a)}{\sqrt{p^{-}}} )}\right)
\label{(4.17)}
\end{equation}
\textbf{(iv)} If $\alpha _{21} ^{\prime }=0$ and $\alpha
_{11}^{\prime }= 0$, then
\begin{equation}
w(\mu )=\Delta_{24}\alpha _{10}^{\prime }\alpha _{20}^{\prime }%
\mu^{4}\cos\left[ \frac{\mu\left( c-a\right)}{\sqrt{p^{-}}} \right]
\cos \left[ \frac{\mu\left( b-c\right)}{\sqrt{p^{+}}} \right]
+O\left( \left| \mu^{3}\right| e^{\left| t\right|
(\frac{(b-c)}{\sqrt{p^{+}}}+\frac{(c-a)}{\sqrt{p^{-}}} )}\right)
\label{(4.18)}
\end{equation}
Now we are ready to derived the needed asymptotic formulas for
eigenvalues and  eigenfunctions.
\begin{thm}
The boundary-value-transmission problem $(\ref{1})$-$(\ref{4})$ has
an precisely numerable many real eigenvalues, whose behavior may be
expressed by two sequence $\left\{ \mu _{n,1}\right\} $ and $\left\{
\mu _{n,2}\right\} $ with following asymptotic as $n\rightarrow \infty $ \\
\textbf{(i)} If $\alpha _{21} ^{\prime }\neq 0$ and $\alpha
_{11}^{\prime }\neq 0$, then
\begin{equation} \label{(5.1)}
\mu_{n,1}=\frac{\sqrt{p^{-}}(n-3)\pi}{(c-a)} +O\left(
\frac{1}{n}\right),  \ \mu_{n,2}=\frac{\sqrt{p^{+}}n\pi}{(b-c)}
+O\left(\frac{1}{n}\right),
\end{equation}
\textbf{(ii)} If $\alpha _{21} ^{\prime }\neq  0$ and $\alpha
_{11}^{\prime }= 0$, then
\begin{equation}\label{(5.2)}
\mu_{n,1}=\frac{\sqrt{p^{-}}(2n+1)\pi}{2(c-a)} +O\left(
\frac{1}{n}\right), \ \mu_{n,2}=\frac{\sqrt{p^{+}}(n-2)\pi}{(b-c)}
+O\left(
\frac{1}{n}\right),%
\end{equation}
\textbf{(iii)} If $\alpha _{21} ^{\prime }= 0$ and $\alpha
_{11}^{\prime }\neq 0$, then
\begin{equation}\label{(5.3)}
\mu_{n,1}=\frac{\sqrt{p^{-}}(n-2)\pi}{(c-a)} +O\left(
\frac{1}{n}\right) , \
\mu_{n,2}=\frac{\sqrt{p^{+}}(2n+1)\pi}{2(b-c)} +O\left(
\frac{1}{n}\right),
\end{equation}
\textbf{(iv)} If $\alpha _{21} ^{\prime }= 0$ and $\alpha
_{11}^{\prime }= 0$, then
\begin{equation} \label{(5.4)}
\mu_{n,1}=\frac{\sqrt{p^{-}}(2n-3)\pi}{2(c-a)}
+O\left(\frac{1}{n}\right), \  \mu_{n,2}=\frac{\sqrt{p^{+}}(2n+1)\pi
}{2(b-c)}+O\left( \frac{1}{n}\right)
\end{equation}
\end{thm}
\begin{proof}
\end{proof}
Using this asymptotic expression of eigenvalues we can easily obtain
the corresponding asymptotic expressions for eigenfunctions of the
problem $(\ref{1})$-$(\ref{4})$.
 Recalling that $\varphi_{\mu _{n,i}}(x)$ is an eigenfunction
 according to the eigenvalue $\mu _{n},$ and by putting (\ref{(5.1)}) in the (\ref{(4.3)})-(\ref{(4.4)}) for $k=0,1$
 and denoting the corresponding  eigenfunction as
\begin{eqnarray*}
\widetilde{\varphi}_{n,i}=\{
\begin{array}{c}
\varphi^{-} _{\mu _{n,i}}(x) \ \textrm{for }x\in \lbrack a,c) \\
\varphi^{+} _{\mu _{n,i}}(x,) \  \textrm{for }x\in (c,b]
\end{array}
\label{(3.16)}
\end{eqnarray*}
 we get the following cases If $\alpha _{21} ^{\prime }\neq 0$ and $\alpha _{11}\neq
0$
\begin{eqnarray*}
\widetilde{\varphi} _{n,1}(x)=\left\{
\begin{array}{ll}
\begin{array}{l}
-\alpha'_{11}p^{-}\left[\frac{(n-3)\pi}{(c-a)}\right]^{2}\cos \left[
 \frac{(n-3)\pi
 (x-a)}{(c-a)}\right]+O\left(n\right),
\end{array}
\begin{array}{l}
\textrm{for }x\in \lbrack a,c)
\end{array}
\\
\frac{\Delta_{24}\alpha' _{11}p^{-}}{\Delta_{12}}\left[
\frac{(n-3)\pi}{(c-a)}\right]^{3} \sin \left[
  (n-3)\pi\right]\cos \left[\frac{\sqrt{p^{-}}
  (n-3)\pi(x-c)}{\sqrt{p^{+}}(c-a)}\right]
\nonumber
\\
+O\left(n^{2} \right), \textrm{ for }x\in (c,b]
\end{array}
\right.
\end{eqnarray*}
and
\begin{eqnarray*}
\widetilde{\varphi} _{n,2}(x)=\left\{
\begin{array}{l}
-\alpha'_{11}p^{+}\left[\frac{n\pi}{(b-c)}\right]^{2}\cos \left[
  \frac{\sqrt{p^{+}}n\pi
 (x-a)}{\sqrt{p^{-}}(b-c)}\right]+O\left(n
\right),\textrm{ \ for }x\in \lbrack a,c)
\\
\frac{\Delta_{24}\alpha'
_{11}}{\Delta_{12}\sqrt{p^{-}}}\left[\frac{\sqrt{p^{+}}n\pi}{(b-c)}\right]^{3}
\sin
\left[\frac{\sqrt{p^{+}}n\pi(c-a)}{\sqrt{p^{-}}(b-c)}\right]\cos
\left[
  \frac{n\pi
 (x-c)}{(b-c)}\right]\nonumber
\\
+O\left( n^{2}\right), \textrm{ \ for }x\in (c,b]
\end{array}
\right.
\end{eqnarray*}
  If $\alpha _{21} ^{\prime }\neq 0$ and $\alpha _{11}=
0$, then
\begin{eqnarray*}
\widetilde{\varphi} _{n,1}(x)=\left\{
\begin{array}{ll}
\begin{array}{l}
-\alpha'_{10}p^{-}\left[\frac{
 (2n+1)\pi}{2(c-a)}\right]\sin \left[\frac{
 (2n+1)\pi(x-a)}{2(c-a)}\right]+O\left(1
\right),
\end{array}
\begin{array}{l}
\textrm{ for }x\in \lbrack a,c)
\end{array}
\\
-\frac{\Delta_{24}\alpha' _{10}p^{-}}{\Delta_{12}}\left[\frac{
 (2n+1)\pi}{2(c-a)} \right]^{2} \cos
\left[\frac{
 (2n+1)\pi}{2}\right]\cos \left[\frac{\sqrt{p^{-}}(2n+1)\pi(x-c)}{2\sqrt{p^{+}}(c-a)}\right]\nonumber \\
 +O\left(n
\right),  \textrm{ for }x\in (c,b]
\end{array}
\right.
\end{eqnarray*}
and
\begin{eqnarray*}
\widetilde{\varphi} _{n,2}(x)=\left\{
\begin{array}{l}
-\alpha'_{10}\sqrt{p^{-}}\left[\frac{\sqrt{p^{+}}(n-2)\pi}{(b-c)}\right]\sin
\left[\frac{\sqrt{p^{+}}
  (n-2)\pi
 (x-a)}{\sqrt{p^{-}}(b-c)}\right]+O\left(1
\right),\textrm{ \ for }x\in \lbrack a,c)
\\
- \frac{\Delta_{24}\alpha' _{10}p^{+}}{\Delta_{12}} \left[\frac{
  (n-2)\pi}{(b-c)}\right]^{2} \cos\left[ \frac{\sqrt{p^{+}}
  (n-2)\pi
 (c-a)}{\sqrt{p^{-}}(b-c)}\right]\cos \left[
  \frac{(n-2)\pi
 (x-c)}{(b-c)}\right]
\nonumber \\
+O\left( n\right),\textrm{ for }x\in (c,b]
\end{array}\right.
\end{eqnarray*}
 If $\alpha _{21} ^{\prime }= 0$ and $\alpha _{11}\neq
0$, then
\begin{eqnarray*}
\widetilde{\varphi} _{n,1}(x)=\left\{
\begin{array}{ll}
\begin{array}{l}
-\alpha'_{11}p^{-}\left[\frac{
 (n-2)\pi}{(c-a)} \right]^{2} \cos \left[
  \frac{
 (n-2)\pi(x-a)}{(c-a)}\right]+O\left(n
\right),
\end{array}
\begin{array}{l}
\textrm{ for }x\in \lbrack a,c)
\end{array}
\\
\frac{-\Delta_{24}\alpha' _{11}p^{-}}{\Delta_{12}} \left[\frac{
 (n-2)\pi}{(c-a)} \right]^{3} \sin \left[
  (n-2)\pi\right]\cos \left[
  \frac{\sqrt{p^{-}}(n-2)\pi
 (x-c)}{\sqrt{p^{+}}(c-a)}\right]
\nonumber
\\
+O\left(n^{2} \right),  \textrm{ for }x\in (c,b]
\end{array}
\right.
\end{eqnarray*}
and
\begin{eqnarray*}
\widetilde{\varphi} _{n,2}(x)=\left\{
\begin{array}{l}
-\alpha'_{11}p^{+}\left[\frac{
 (2n+1)\pi}{2(b-c)} \right]^{2} \cos \left[
  \frac{\sqrt{p^{+}}
 (2n+1)\pi(x-a)}{2\sqrt{p^{-}}(b-c)}\right]+O\left(n
\right),\textrm{ \ for }x\in \lbrack a,c)
\\
-\frac{\Delta_{24}\alpha' _{11}}{\Delta_{12}\sqrt{p^{-}}}
\left[\frac{
 \sqrt{p^{+}}(2n+1)\pi}{2(b-c)} \right]^{3} \sin \left[
  \frac{
 \sqrt{p^{+}}(2n+1)\pi(c-a)}{2\sqrt{p^{-}}(b-c)}\right]\cos \left[
  \frac{
(2n+1)\pi(x-c)}{2(b-c)}\right] \nonumber
\\
+O\left(n^{2} \right),\textrm{ for }x\in (c,b]
\end{array}
\right.
\end{eqnarray*}
 If $\alpha _{21} ^{\prime }= 0$ and $\alpha _{11}= 0$,
then
\begin{eqnarray*}
\widetilde{\varphi} _{n,1}(x)=\left\{
\begin{array}{ll}
\begin{array}{l}
-\alpha'_{10}p^{-}\left[\frac{
 (2n-3)\pi}{2(c-a)} \right]\sin \left[
  \frac{
 (2n-3)\pi(x-a)}{2(c-a)}\right]+O\left(1
\right)
\end{array}
\begin{array}{l}
\textrm{ for }x\in \lbrack a,c)
\end{array}
\\
- \frac{ \Delta_{24}p^{-}\alpha' _{10}}{\Delta_{12}}\left[\frac{
 (2n-3)\pi}{2(c-a)} \right]^{2} \cos
  \left[\frac{
 (2n-3)\pi}{2} \right]\cos \left[
  \frac{\sqrt{p^{-}}
 (2n-3)\pi(x-c)}{2\sqrt{p^{+}}(c-a)} \right]
\nonumber
\\
+O\left(1 \right), \ \textrm{ for } \ x\in (c,b]
\end{array}
\right.
\end{eqnarray*}
and
\begin{eqnarray*}
\widetilde{\varphi} _{n,2}(x)=\left\{
\begin{array}{ll}
\begin{array}{l}
-\alpha'_{10}\sqrt{p^{-}}\left[\frac{\sqrt{p^{+}}
(2n+1)\pi}{2(b-c)}\right]\sin\left[\frac{\sqrt{p^{+}}
 (2n+1)\pi(x-a)}{2\sqrt{p^{-}}(b-c)}\right]+O\left(1
\right),
\end{array}
\begin{array}{l}
\textrm{for} \ x\in \lbrack a,c)
\end{array}
\\
- \frac{ \Delta_{24}p^{+}\alpha' _{10}}{\Delta_{12}}\left[\frac{
 (2n+1)\pi}{2(b-c)} \right]^{2} \cos
  \left[\frac{\sqrt{p^{+}}
 (2n+1)\pi(c-a)}{2\sqrt{p^{-}}(b-c)} \right]\cos \left[
  \frac{
 (2n+1)\pi(x-c)}{2(b-c)} \right]\nonumber\\
+O\left(\frac{1}{n} \right), \ \textrm{for} \ x\in (c,b]
\end{array}
\right.
\end{eqnarray*}
All this asymptotic approximations are hold uniformly for $x.$
 \textbf{Example.} Consider the
following simple case of the BVTP's $(\ref{1})-(\ref{3})$
\begin{equation}\label{ex}
-y^{\prime \prime }(x)=\mu^{2}y(x) \ \ \ \ x\in [-\pi,0)\cup(0,\pi]
\end{equation}
\begin{equation}\label{ex1}
y(-\pi)+\mu^{2} y'(-\pi)=0,
\end{equation}
\begin{equation}\label{ex2}
\mu^{2} y(\pi)+y'(\pi)=0,
\end{equation}
\begin{equation}\label{ex3}
y(0-)=2y(+0), \  \ y'(-0)=y'(+0)
\end{equation}

\begin{minipage}{0.5\textwidth}
  \includegraphics[height=5cm,width=10cm]{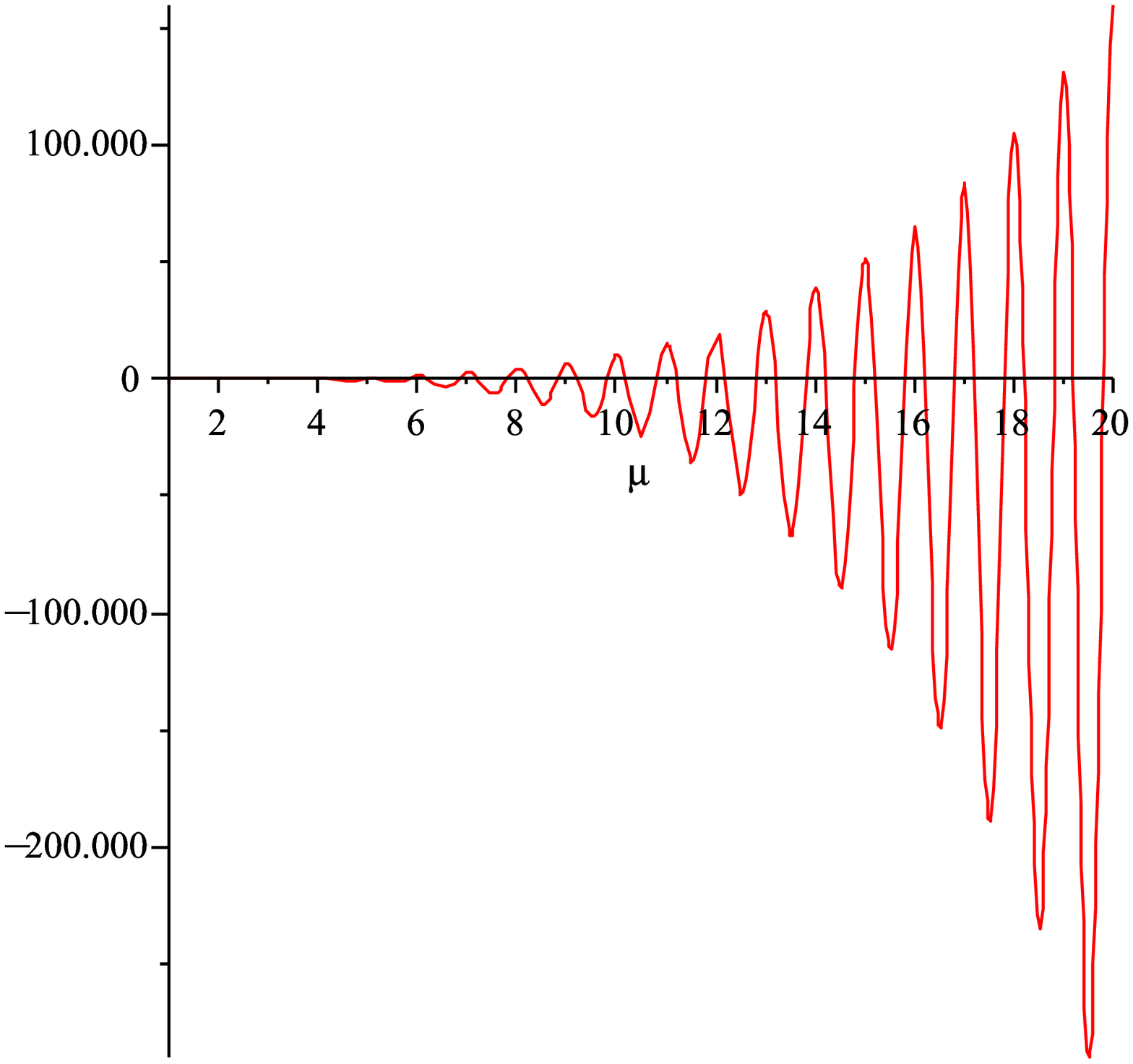}
\verb Figure 1: Graph of the function $w(\mu)$ for real $\mu$
\end{minipage}\\
\begin{minipage}{0.5\textwidth}
  \includegraphics[height=5cm,width=7cm]{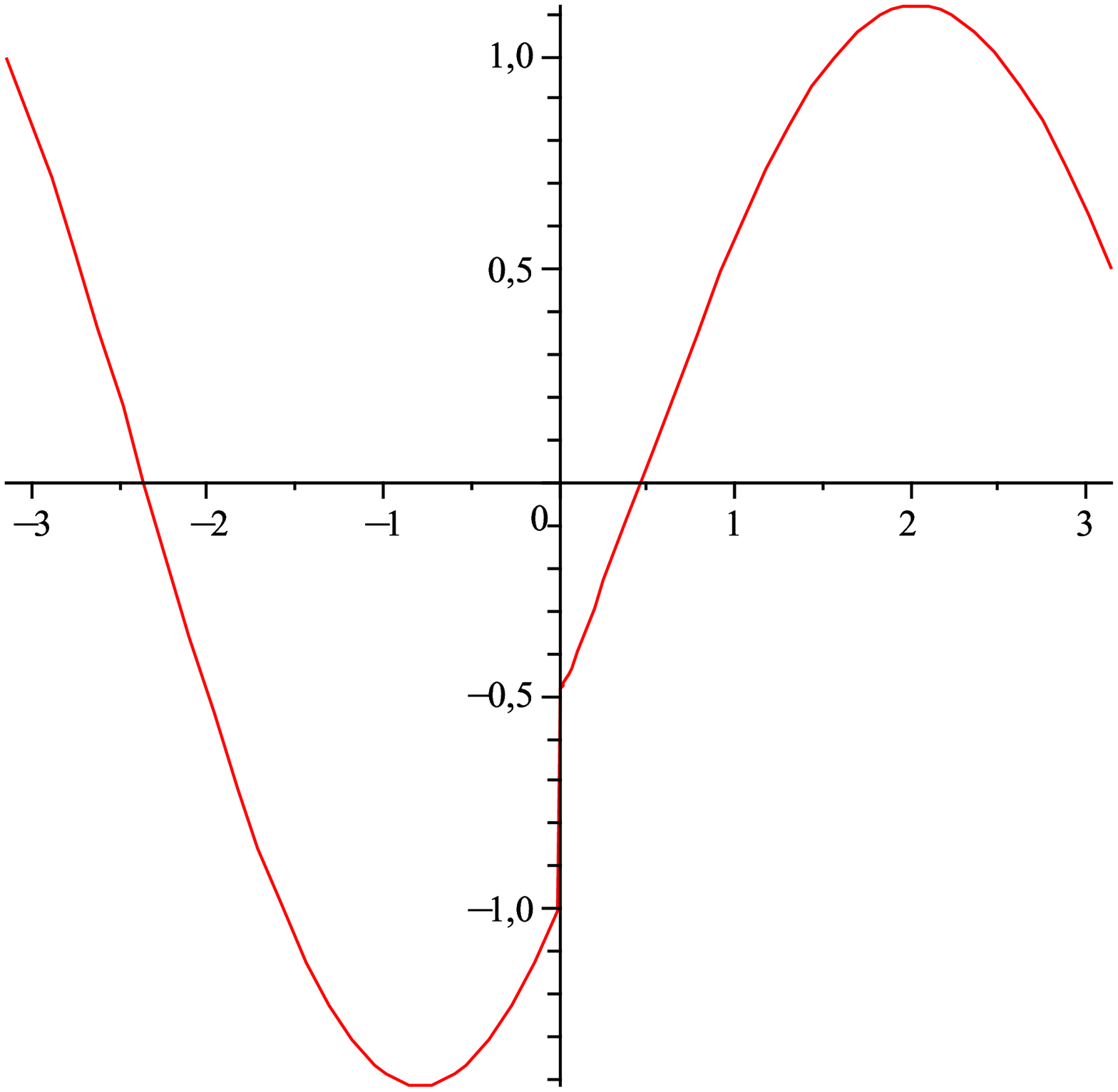}
\verb Figure \ 2:Graph of the solution  $\phi(x,\mu)$  for $\mu=1$
\end{minipage}
\begin{minipage}{0.5\textwidth}
  \includegraphics[height=5cm,width=7cm]{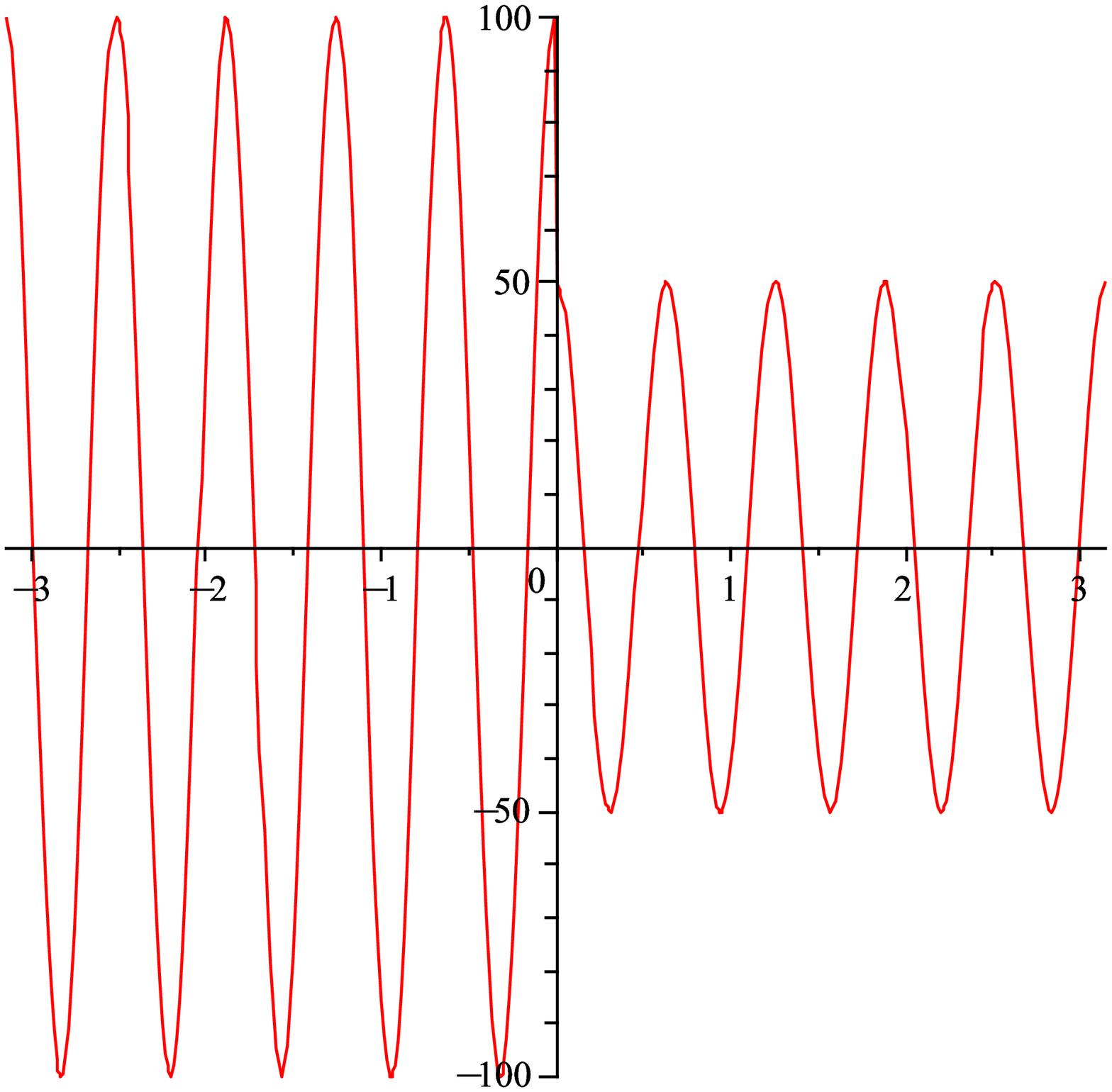}
\verb Figure \ 3:Graph of the solution $\phi(x,\mu)$  for $\mu=10$
\end{minipage}

\section*{Acknowledgement}
The authors grateful for financial support from the research fund of
Gaziosmanpa\c{s}a University under grant no:\ $2012\backslash 126.$

\end{document}